\documentclass[12pt,reqno]{article}

\usepackage[usenames]{color}
\usepackage{amssymb}
\usepackage{graphicx}
\usepackage{amscd}

\usepackage[colorlinks=true,
linkcolor=webgreen,
filecolor=webbrown,
citecolor=webgreen]{hyperref}

\definecolor{webgreen}{rgb}{0,.5,0}
\definecolor{webbrown}{rgb}{.6,0,0}

\usepackage{color}
\usepackage{fullpage}
\usepackage{float}
\usepackage{graphics,amsmath,amssymb}
\usepackage{amsthm}
\usepackage{amsfonts}
\usepackage{latexsym}

\newcommand{\seqnum}[1]{\href{http://oeis.org/#1}{\underline{#1}}}

\begin{document}

	\theoremstyle{plain}
	\newtheorem{theorem}{Theorem}
	\newtheorem{lemma}[theorem]{Lemma}
	\newtheorem{corollary}[theorem]{Corollary}
	\theoremstyle{definition}
	\newtheorem{definition}[theorem]{Definition}
	\newtheorem{example}[theorem]{Example}
	\newtheorem{remark}[theorem]{Remark}

\begin{center}
	\vskip 1cm{\LARGE\bf
		Additional close links between balancing and Lucas-balancing polynomials\\
		\vskip .1in } \vskip 0.5cm \large
Robert Frontczak\footnote{Statements and conclusions made in this article by R. Frontczak are entirely those of the author. They do not necessarily reflect the views of LBBW.}\\
Landesbank Baden-W{\"u}rttemberg,\\
Stuttgart, Germany\\
\href{mailto:robert.frontczak@lbbw.de}{\tt robert.frontczak@lbbw.de}\\
\vskip .2 in

		Taras Goy\\
	Vasyl Stefanyk Precarpathian National University \\
	Ivano-Frankivsk, Ukraine\\
	\href{mailto:taras.goy@pnu.edu.ua}{\tt taras.goy@pnu.edu.ua} \\
	\vskip .2 in
\end{center}

\vskip .2 in

\begin{abstract}
Using generating functions, we derive many  identities involving balancing and Lucas-balancing polynomials. By relating these polynomials to Chebyshev polynomials of the first and second kind, and Fibonacci and Lucas numbers, we offer some presumably new combinatorial identities involving these famous sequences.
\end{abstract}

\section{Introduction and preliminaries}

For any integer $n\geq0$, the balancing polynomials $\big(B_n(x)\big)_{n\geq0}$ and Lucas-balancing polynomials $\big(C_n(x)\bigr)_{n\geq0}$ are defined by the same second-order homogeneous linear recurrence 
\begin{equation}\label{B(x)-def}
	u_n(x) = 6x u_{n-1}(x) - u_{n-2}(x), 
\end{equation}
but with different initial terms $B_0(x)=0$, $B_1(x)=1$ and 
$C_0(x)=1$, $C_1(x)=3x$. These polynomials have been introduced as a natural extension of the popular balancing and Lucas-balancing numbers $B_n$ and $C_n$, respectively, which were firstly studied in \cite{Behera}. Obviously, $B_n=B_n(1)$ and $C_n=C_n(1)$.  Sequences $\big(B_n\big)_{n\geq 0}$ and $\big(C_n\big)_{n\geq 0}$ are indexed in the On-Line Encyclopedia of Integer Sequences \cite{Sloane} (see entries \seqnum{A001109} and \seqnum{A001541}, respectively). 

The closed forms which are also called Binet's formulas for balancing and Lucas-balancing polynomials are given by 
\begin{equation}\label{Binet}
B_n(x) = \frac{\lambda^n (x) - \lambda^{-n} (x)}{2\sqrt{9x^2-1}},\qquad
C_n(x) = \frac{\lambda^n (x) + \lambda^{-n} (x)}{2},
\end{equation}
where $\lambda (x)=3x + \sqrt{9x^2-1}$.

For $n\geq1$, the  balancing and Lucas-balancing polynomials are given explicitly \cite{Patel,Ray-UMJ} by
\begin{gather*}\label{Comb-expr-B(x)}
B_{n}(x) = \sum_{k=0}^{\left\lfloor{\frac{n-1}{2}}\right \rfloor} (-1)^k \binom{n-1-k}{k} (6x)^{n-1-2k}, \quad 
C_{n}(x)=\frac{n}{2}\sum_{k=0}^{\left \lfloor{\frac{n}{2}}\right \rfloor}  \frac{(-1)^k}{n-k} \binom{n-k}{k} (6x)^{n-2k}.
\end{gather*}

The first polynomials  are
\begin{gather*}
B_0(x)=0,\quad\,\, B_1(x)=1,\quad\,\, B_2(x)=6x,\quad\,\, B_3(x)=36x^2-1,\\ B_4(x)=216x^3-12x, \,\,\quad B_5(x)=1296x^4-108x^2+1,
\end{gather*}
and
\begin{gather*}
C_0(x)=1,\,\,\quad C_1(x)=3x,\,\,\quad C_2(x)=18x^2-1,\,\,\quad C_3(x)=108x^3-9x,\\ 
C_4(x)=648x^4-72x^2+1,\,\,\quad C_5(x)=3888x^5-540x^3+15x.
\end{gather*}

These polynomials have been studied extensively in different contexts and a variety of interesting	results about them have been uncovered \cite{Frontczak-AMS,Frontczak-IJMA19,AMI,Kim1,Kim2,Meng, Ray-UMJ}. For example, in \cite{Frontczak-AMS}, the first author established direct connections of the polynomials $B_n(x)$ and $C_n(x)$ with Fibonacci numbers, Lucas numbers and Chebyshev and Legendre polynomials. By using combinatorial methods, Meng derived some symmetry identities of the structural properties of balancing numbers and balancing polynomials \cite{Meng}. In \cite{Kim1,Kim2}, the authors study sums of finite products of balancing and Lucas-balancing polynomials and represent them in terms of nine orthogonal polynomials. In \cite{Ray-UMJ}, Ray studied the sequences obtained by differentiating the balancing polynomials and presented some relations between the balancing polynomials and their derivatives.

In the present study, we derive new identities for polynomials $B_n(x)$ and $C_n(x)$. Evaluating these identities at specific points, we can also establish some interesting combinatorial identities as special cases, especially those with Fibonacci and Lucas numbers. Our approach is in the spirit of \cite{Frontczak-FQ,Frontczak-CMP}.

\section{Balancing polynomial relations using ordinary \\ generating functions}

To establish our main results, we will find the ordinary (non-exponential) generating functions for the sequences in question.  We will make use of the following result \cite{Mezo} to compute the ordinary generating functions for balancing, Lucas-balancing polynomials and their odd (even) indexed companions. 
\begin{lemma}  The second-order polynomial recurrence $u_n(x) = pu_{n-1}(x) + qu_{n-2}(x)$, $n\geq2$, $p^2+4q\ne0$, with initial terms $u_0(x)$ and $u_1(x)$  has generating function
\begin{gather*}
\sum_{n\geq0}u_n(x)z^n=\frac{u_0(x)+\bigl(u_1(x)-pu_0(x)\bigr)z}{1-pz-qz^2}, 
\end{gather*}
while for odd (even) indexed sequences 
\begin{gather*}
\sum_{n\geq0}u_{2n+1}(x)z^n=\frac{u_1(x)+\bigl(pqu_0(x)-qu_1(x)\bigr)z}{1-(p^2+2q)z+q^2z^2},\\
\sum_{n\geq0}u_{2n}(x)z^n=\frac{u_0(x)+\bigl(u_2(x)-(p^2+2q)u_0(x)\bigr)z}{1-(p^2+2q)z+q^2z^2}.
\end{gather*}
\end{lemma}

From the above lemma we obtain ordinary generating functions of the sequences $B_n(x)$, $B_{2n+1}(x)$, and $B_{2n}(x)$ as follows
\begin{gather}
b(x,z)=\sum_{n\geq0} B_n(x)z^n=\frac{z}{1-6xz+z^2}, \label{gf-B}\\
b_1(x,z)=\sum_{n\geq0}B_{2n+1}(x)z^{n}=\frac{1+z}{1-(36x^2-2)z+z^2}, \nonumber\\
b_2(x,z)=\sum_{n\geq0}B_{2n}(x)z^{n}=\frac{6xz}{1-(36x^2-2)z+z^2}. \label{gf-B2}
\end{gather}

In the similar manner, we conclude that ordinary  generating functions of the sequences $C_n(x)$, $C_{2n+1}(x)$ and $C_{2n}(x)$ can be derived as
\begin{gather}
c(x,z)=\sum_{n\geq0} C_n(x)z^n=\frac{1-3xz}{1-6xz+z^2}, \nonumber\\
c_1(x,z)=\sum_{n\geq0} C_{2n+1}(x)z^{n}=\frac{3x-3xz}{1-(36x^2-2)z+z^2}, \label{gf-C1}\\
c_2(x,z)=\sum_{n\geq0} C_{2n}(x)z^{n}=\frac{1+(1-18x^2)z}{1-(36x^2-2)z+z^2}. \label{gf-C2}
\end{gather}

We present our first findings in two theorems, which provide some relations between balancing and  Lucas-balancing polynomials  using  their respective ordinary generating functions.
\begin{theorem} For $n\geq1$, the following formulas hold \label{Theo1_BC-gf}
\begin{gather}
B_n(x)-3xB_{n-1}(x)=C_{n-1}(x),\nonumber\\[2pt]
3x\bigl(B_{2n+1}(x)-B_{2n-1}(x)\bigr)=C_{2n+1}(x)+C_{2n-1}(x),\nonumber\\[2pt]
B_{2n}(x)-(18x^2-1)B_{2(n-1)}(x)=6xC_{2(n-1)}(x),\label{B2-C2}\\[2pt]
B_{2n+1}(x)-(18x^2-1)B_{2n-1}(x)=C_{2n}(x)+C_{2(n-1)}(x),\nonumber\\[2pt]
3x\bigl(B_{2n}(x)-B_{2(n-1)}(x)\bigr)=6xC_{2n-1}(x).\nonumber
\end{gather}
\end{theorem}
\begin{proof} All stated identities can be proved directly using Binet's formulas \eqref{Binet}. We present a proof based on generating functions. We will prove formula \eqref{B2-C2}; the others may also be shown in a similar manner.
	
By \eqref{gf-B2} and \eqref{gf-C2}, we obtain
\begin{equation*}
\bigl(1-(18x^2-1)z\bigr)b_2(x,z)=
6xzc_2(x,z).
\end{equation*}

Expanding both sides of the last equation as a power series in $z$, we then have 
\begin{gather*}
\sum_{n\geq0}B_{2n}(x)z^{n}-(18x^2-1)\sum_{n\geq0}B_{2n}(x)z^{n+1}=6x\sum_{n\geq0}C_{2n}(x)z^{n+1}
\end{gather*}
or, equivalently,
\begin{gather*}
\sum_{n\geq1}B_{2n}(x)z^{n}-(18x^2-1)\sum_{n\geq1}B_{2(n-1)}(x)z^{n}=6x\sum_{n\geq1}C_{2(n-1)}(x)z^{n}.
\end{gather*}
Comparing the coefficients on both sides, we get \eqref{B2-C2}.
\end{proof}

For convention, throughout this paper the empty sums are evaluated to $0$.
\begin{theorem} \label{Theo2_BC-gf} For $n\geq1$, the following formulas hold: 
	\begin{gather}
	3x\bigl(B_{n}(x)-B_{n-1}(x)\bigr)=C_{2n-1}(x)-(36x^2-6x-2)\sum_{k=1}^{n-1}B_k(x)C_{2(n-k)-1}(x),\label{B-C1}\\
	B_{n}(x)-(18x^2-1)B_{n-1}(x)
	=C_{2(n-1)}(x)-(36x^2-6x-2)\sum_{k=1}^{n-1}B_k(x)C_{2(n-k-1)}(x),\nonumber\\
	B_{2n+1}(x)-3xB_{2n-1}(x)
	=C_{n}(x)+C_{n-1}(x)-(36x^2-6x-2)\sum_{k=0}^{n-1}B_{2k+1}(x)C_{n-k-1}(x),\nonumber\\
		B_{2n}(x)-3xB_{2(n-1)}(x)=6xC_{n-1}(x)-(36x^2-6x-2)\sum_{k=1}^{n-1}B_{2k}(x)C_{n-k-1}(x).\nonumber
	\end{gather}
\end{theorem}
\begin{proof} We prove only \eqref{B-C1}. Proceeding as in the proof of Theorem \ref{Theo1_BC-gf} above, using \eqref{gf-B} and \eqref{gf-C1}, we deduce that
\begin{equation*}
zc_1(x,z)-3x(1-z)b(x,z)=(36x^2-6x-2)zb(x,z)c_1(x,z).
\end{equation*}
Expanding both sides of the last equation as a power series in $z$ yields
\begin{gather*}
\sum_{n\geq0}C_{2n+1}(x)z^{n+1}-3x\sum_{n\geq0}B_{n}(x)z^{n}+3x\sum_{n\geq0}B_{n}(x)z^{n+1}\\
=(36x^2-6x-2)\sum_{n\geq0}\sum_{k=0}^{n}B_k(x)C_{2(n-k)+1}z^{n+1},\\
\sum_{n\geq1}C_{2n-1}(x)z^{n}-3x\sum_{n\geq1}B_{n}(x)z^{n}+3x\sum_{n\geq1}B_{n-1}(x)z^{n}\\
=(36x^2-6x-2)\sum_{n\geq1}\sum_{k=0}^{n-1}B_k(x)C_{2(n-k)-1}z^{n}.
\end{gather*}

Comparing the coefficients on both sides implies the stated formula. 
\end{proof}

\section{Balancing relations using exponential generating\\ functions}

In this section, we will use the structure of the exponential generating functions to prove our results. In this case, we will use significantly the following lemma from \cite{Mezo}. 
\begin{lemma} The recurrence of type 
	$u_n(x) = pu_{n-1}(x) + qu_{n-2}(x)$, $n\geq2$,
	where $p^2+4q\ne0$ and $u_0(x)$ and $u_1(x)$ are the initial terms, has the exponential generating function
	\begin{gather*}
	\sum_{n\geq0}u_n(x)\frac{z^n}{n!} = \frac{e^{\frac{p}{2}z}}{\Delta}\Bigl((u_1(x)-\beta u_0(x))e^{\frac{\Delta}{2}z}-(u_1(x)-\alpha u_0(x))e^{-\frac{\Delta}{2}z}\Bigr),  \label{egf-Mezo}
	\end{gather*}
	while for odd and even indexed sequences 
	\begin{gather*}
	\sum_{n\geq0}u_{2n+1}(x)\frac{z^n}{n!}=\frac{e^{\frac{p^2+2q}{2}z}}{p\Delta}\Bigl((u_3(x)-\sigma u_1(x))e^{\frac{p\Delta}{2}z}-(u_3(x)-\rho u_0(x))e^{-\frac{p\Delta}{2}z}\Bigr),  \label{egf1-Mezo}\\
	\sum_{n\geq0}u_{2n}(x)\frac{z^n}{n!}=\frac{e^{\frac{p^2+2q}{2}z}}{p\Delta}\Bigl((u_2(x)-\sigma u_0(x))e^{\frac{p\Delta}{2}z}-(u_2(x)-\rho u_0(x))e^{-\frac{p\Delta}{2}z}\Bigr),\label{egf2-Mezo}
	\end{gather*}
	where $\Delta=\sqrt{p^2+4q}$,  $\alpha=\frac{p+\Delta}{2}$, $\beta=\frac{p-\Delta}{2}$, $\rho=\frac{p^2+2q+p\Delta}{2}$, and $\sigma=\frac{p^2+2q-p\Delta}{2}$.
\end{lemma}

From lemma above it can be shown fairly easily that the exponential generating functions of the balancing polynomials and their odd (even) indexed  companions are given by
\begin{gather}
b^*(x,z)=\sum_{n\geq0} B_n(x)\frac{z^n}{n!}=\frac{e^{3xz}}{\sqrt{9x^2-1}}\sinh\!\big(\sqrt{9x^2-1}\,z\big), \label{egf-B}\\
b^*_1(x,z)=\sum_{n\geq0}B_{2n+1}(x)\frac{z^n}{n!}\nonumber\\
=\frac{e^{(18x^2-1)z}}{\sqrt{9x^2-1}}\left(3x\sinh\!\big(6x\sqrt{9x^2-1}\,z\big)+\sqrt{9x^2-1}\cosh\!\big(6x\sqrt{9x^2-1}\,z\big)\right), \label{egf-B1}\\
b^*_2(x,z)=\sum_{n\geq0}B_{2n}(x)\frac{z^n}{n!}=\frac{e^{(18x^2-1)z}}{\sqrt{9x^2-1}}\sinh\!\big(6x\sqrt{9x^2-1}\,z\big). \label{egf-B2}
\end{gather}

In the similar manner, we obtain the exponential generating functions of the Lucas-balancing polynomial sequences:
\begin{gather}
c^*(x,z)=\sum_{n\geq0} C_n(x)\frac{z^n}{n!}=e^{3xz}\cosh\!\big(\sqrt{9x^2-1}\,z\big), \label{egf-C}\\
c^*_1(x,z)=\sum_{n\geq0} C_{2n+1}(x)\frac{z^n}{n!}\nonumber\\
=e^{(18x^2-1)z}\left(3x\cosh\!\big(6x\sqrt{9x^2-1}\,z\big)+\sqrt{9x^2-1}\sinh\!\big(6x\sqrt{9x^2-1}\,z\big)\right), \label{egf-C1}
\end{gather}
\begin{gather}
c^*_2(x,z)=\sum_{n\geq0} C_{2n}(x)\frac{z^n}{n!}=e^{(18x^2-1)z}\cosh\!\big(6x\sqrt{9x^2-1}\,z\big). \label{egf-C2}
\end{gather}

Next, we present some new relations between balancing and Lucas-balancing polynomials involving binomial coefficients. 
\begin{theorem}\label{Theo1_BC-egf} For $n\geq1$, we have  
		\begin{gather}
	\sum_{k=1}^n{n\choose k}\frac{1+(-1)^{n-k}}{\bigl(\sqrt{9x^2-1}\bigr)^{k-1}}B_{k}(x)=	\sum_{k=0}^n{n\choose k}\frac{1-(-1)^{n-k}}{\bigl(\sqrt{9x^2-1}\bigr)^k}C_k(x),\label{B-C-binom}\\
	\sum_{k=0}^n\!{n\choose k}\!\frac{\lambda(x)+(-1)^{n-k}\lambda^{-1}(x)}{\bigl(6x\sqrt{9x^2-1}\bigr)^{k-1}}B_{2k+1}(x)
	=6x\sum_{k=0}^n\!{n\choose k}\!\frac{\lambda(x)-(-1)^{n-k}\lambda^{-1}(x)}{\bigl(6x\sqrt{9x^2-1}\bigr)^k}C_{2k+1}(x),\label{B1-C1-binom}\\
	\sum_{k=1}^n{n\choose k}\frac{1+(-1)^{n-k}}{\bigl(6x\sqrt{9x^2-1}\bigr)^{k-1}}B_{2k}(x)=
	6x\sum_{k=1}^n{n\choose k}\frac{1-(-1)^{n-k}}{\bigl(6x\sqrt{9x^2-1}\bigr)^k}C_{2k}(x),\label{B2-C2-binom}\\
		\sum_{k=0}^n{n\choose k}\frac{1+(-1)^{n-k}}{\bigl(6x\sqrt{9x^2-1}\bigr)^{k-1}}
	B_{2k+1}(x)=6x\sum_{k=0}^n{n\choose k}\frac{\lambda(x)-(-1)^{n-k}\lambda^{-1}(x)}{\bigl(6x\sqrt{9x^2-1}\bigr)^{k}}C_{2k}(x),\nonumber\\
	\sum_{k=1}^n{n\choose k}\frac{\lambda(x)+(-1)^{n-k}\lambda^{-1}(x)}{\bigl(6x\sqrt{9x^2-1}\bigr)^{k-1}}B_{2k}(x)=6x\sum_{k=0}^n{n\choose k}\frac{1-(-1)^{n-k}}{\bigl(6x\sqrt{9x^2-1}\bigr)^k}C_{2k+1}(x),\nonumber
		\end{gather}
	where $\lambda(x)=3x+\sqrt{9x^2-1}$.
\end{theorem}
\begin{proof} We will prove \eqref{B-C-binom}. In view of \eqref{egf-B} and \eqref{egf-C}, we have  
\begin{equation*}
\sqrt{9x^2-1}\cosh\!\big(\sqrt{9x^2-1}z\big)\,b^*(x,z)=\sinh\!\big(\sqrt{9x^2-1}z\big)\,c^*(x,z)
\end{equation*}
or, equivalently,
\begin{gather*}
\sqrt{9x^2-1}\sum_{n\geq0}B_n(x)\frac{z^n}{n!}
\sum_{n\geq0}\big(\sqrt{9x^2-1}\big)^n\big(1+(-1)^n\big)\frac{z^n}{n!}\\
=\sum_{n\geq0}C_n(x)\frac{z^n}{n!}
\sum_{n\geq0}\big(\sqrt{9x^2-1}\big)^n\big(1-(-1)^n\big)\frac{z^n}{n!},\\
\sqrt{9x^2-1}\sum_{n\geq0}\sum_{k=0}^n{n\choose k}B_k(x)  \big(\sqrt{9x^2-1}\big)^{n-k}\big(1+(-1)^{n-k}\big)\frac{z^n}{n!}\\
=\sum_{n\geq0}\sum_{k=0}^n{n\choose k}C_k(x)\big(\sqrt{9x^2-1}\big)^{n-k}\big(1-(-1)^{n-k}\big)\frac{z^n}{n!}.
\end{gather*}

Comparing the coefficients on both sides yields \eqref{B-C-binom}. 

The formulas \eqref{B1-C1-binom} and \eqref{B2-C2-binom} follow from the relations
\begin{gather*}
\sqrt{9x^2-1}\left(3x\cosh\!\big(6x\sqrt{9x^2-1}\big)+\sqrt{9x^2-1}\sinh\!\big(6x\sqrt{9x^2-1}\big)\right)b_1^*(x,z)\\
=\left(3x\cosh\!\big(6x\sqrt{9x^2-1}\big)-\sqrt{9x^2-1}\sinh\!\big(6x\sqrt{9x^2-1}\big)\right)c_1^*(x,z)
\end{gather*}
and
\begin{gather*}
\sqrt{9x^2-1}\cosh\!\big(6x\sqrt{9x^2-1}\big)\,b_2^*(x,z)
=\sinh\!\big(6x\sqrt{9x^2-1}\big)\,c_2^*(x,z),
\end{gather*}
that one can obtain from \eqref{egf-B1},  \eqref{egf-C1} and \eqref{egf-B2},  \eqref{egf-C2}, respectively. The proof of the other formulas is similar.
\end{proof}

In the next theorem we list a range of more advanced further relations for  balancing and Lucas-balancing polynomials that can be derived in a similar manner. 
\begin{theorem}\label{Theo2_BC-egf} For $n\geq1$, we have 
\begin{gather}
\sum_{k=1}^n{n\choose k}\bigl(\lambda(x)+(-1)^{n-k}\lambda^{-1}(x)\bigr)\left(\frac{18x^2-1}{18x^2\sqrt{9x^2-1}}\right)^{k-1}B_{k}(x)\nonumber\\
	=\left(\frac{18x^2-1}{18x^2}\right)^{n-1}\sum_{k=0}^n{n\choose k}\bigl(1-(-1)^{n-k}\bigr)\left(\frac{3x}{(18x^2-1)\sqrt{9x^2-1}}\right)^{k}C_{2k+1}(x),\label{Theo6}\\
	\sum_{k=1}^n{n\choose k}\bigl(1+(-1)^{n-k}\bigr)\left(\frac{18x^2-1}{18x^2\sqrt{9x^2-1}}\right)^{k-1}B_{k}(x)\nonumber\\
	=\left(\frac{18x^2-1}{18x^2}\right)^{n-1}\sum_{k=0}^n{n\choose k}\bigl(1-(-1)^{n-k}\bigr)\left(\frac{3x}{(18x^2-1)\sqrt{9x^2-1}}\right)^{k}C_{2k}(x),\nonumber\\
	\sum_{k=0}^n{n\choose k}\bigl(1+(-1)^{n-k}\bigr)\left(\frac{3x}{(18x^2-1)\sqrt{9x^2-1}}\right)^{k-1}B_{2k+1}(x)\nonumber\\
	=\left(\frac{18x^2}{18x^2-1}\right)^{n-1}6x\sum_{k=0}^n{n\choose k}\bigl(\lambda(x)-(-1)^{n-k}\lambda^{-1}(x)\bigr)\left(\frac{18x^2-1}{18x^2\sqrt{9x^2-1}}\right)^{k}C_{k}(x),\nonumber\\
		\sum_{k=1}^n{n\choose k}\bigl(1+(-1)^{n-k}\bigr)\left(\frac{3x}{(18x^2-1)\sqrt{9x^2-1}}\right)^{k-1}B_{2k}(x)\nonumber\\
	=\left(\frac{18x^2}{18x^2-1}\right)^{n-1}6x\sum_{k=0}^n{n\choose k}\bigl(1-(-1)^{n-k}\bigr)\left(\frac{18x^2-1}{18x^2\sqrt{9x^2-1}}\right)^{k}C_{k}(x),\nonumber
	\end{gather}
where $\lambda(x)=3x+\sqrt{9x^2-1}$.
\end{theorem}
\begin{proof}  
Formula \eqref{Theo6} follows from the functional relation
\begin{gather*}
\Big(3x\cosh\!\big(6x\sqrt{9x^2-1}z\big)+\sqrt{9x^2-1}\sinh\!\big(6x\sqrt{9x^2-1}z\big)\Big)b^*\!\left(x,\frac{18x^2-1}{3x}z\right)\\
=\sinh\!\left(\frac{(18x^2-1)(\sqrt{9x^2-1})}{3x}z\right)c_1^*(x,z),
\end{gather*}
writing in terms of power series, and collecting terms.		

The proof of the other formulas is similar. We omit the corresponding details. 
\end{proof}
	
\section{Chebyshev polynomial relations via balancing and Lucas-balancing polynomials}

As usual, Chebyshev polynomials $\big(T_n(x)\big)_{n\geq0}$ of the first kind and the Chebyshev polynomials $\big(U_n(x)\big)_{n\geq0}$ of the second kind can be defined by the recurrence \cite{Mason} 
\begin{equation}\label{Cheb-Recur}
W_{n}(x)=2xW_{n-1}(x)-W_{n-2}(x),\quad n\geq2,
\end{equation} 
but with different initial terms $T_0(x)=1$, $T_1(x)=x$ and $U_0(x)=1$, $U_1(x)=2x$. 

Many properties of Chebyshev polynomials can be obtained immediately from the following relations initially observed by the first author in \cite{Frontczak-AMS}.
\begin{lemma}\label{Lemma-Bal-Cheb} 
For $n\geq1$, the following relations hold:
\begin{equation}\label{Bal-Cheb}
B_n\!\left(\frac{x}{3}\right)=U_{n-1}(x), \quad C_n\!\left(\frac{x}{3}\right)=T_{n}(x).
\end{equation}
\end{lemma}

Applying \eqref{Bal-Cheb} to Theorems \ref{Theo1_BC-gf} and \ref{Theo2_BC-gf} yields the following Chebyshev polynomial relations.
\begin{corollary} For $n\geq1$,
\begin{gather*}
T_{2n-1}(x) - 2(2x^2-x-1)\sum_{k=1}^{n-1}U_{k-1}(x)T_{2n-2k-1}(x) = x(U_{n-1}(x)-U_{n-2}(x)),\\
T_{2n-2}(x) - 2(2x^2-x-1)\sum_{k=1}^{n-1}U_{k-1}(x)T_{2(n-k-1)}(x) = U_{n-1}(x)-(2x^2-1)U_{n-2}(x),\\
T_{n}(x) + T_{n-1}(x) + 2(2x^2-x-1)\sum_{k=1}^{n-1}U_{2k} (x)T_{n-k-1}(x) = U_{2n}(x)-xU_{2(n-1)}(x),\\
2xT_{n-1}(x) +2(2x^2-x-1)\sum_{k=1}^{n-1}U_{2k-1}(x)T_{n-k-1}(x) = U_{2n-1}(x)-xU_{2n-3}(x).
\end{gather*}
\end{corollary}

By virtue of \eqref{Bal-Cheb}, from Theorems \ref{Theo1_BC-egf} and \ref{Theo2_BC-egf} we have the following summation formulas involving Chebyshev polynomials and binomial coefficients. 
\begin{corollary} Let $\omega(x)=x+\sqrt{x^2-1}$. Then for $n\geq1$ the following relations hold:
	\begin{gather*}
	\sum_{k=1}^n{n\choose k}\frac{U_{k-1}(x)}{\big(\sqrt{x^2-1}\big)^{k-1}}\bigl(1+(-1)^{n-k}\bigr) = \sum_{k=0}^n{n\choose k}\frac{T_{k}(x)}{\big(\sqrt{x^2-1}\big)^k}\bigl(1-(-1)^{n-k}\bigr) ,\\
		\sum_{k=1}^n{n\choose k}\frac{U_{2k}(x)}{\big(2x\sqrt{x^2-1}\big)^{k-1}}\bigl(\omega(x)+(-1)^{n-k}\omega^{-1}(x)\bigr)\\
	=2x\sum_{k=0}^n{n\choose k}\frac{T_{2k+1}(x)}{\big(2x\sqrt{x^2-1}\big)^k}\bigl(\omega(x)-(-1)^{n-k}\omega^{-1}(x)\bigr),\\
		\sqrt{x^2-1}\sum_{k=1}^n{n\choose k}\frac{U_{2k-1}(x)}{\big(2x\sqrt{x^2-1}\big)^{k}}\bigl(1+(-1)^{n-k})	=\sum_{k=0}^n{n\choose k}\frac{T_{2k}(x)}{\big(2x\sqrt{x^2-1}\big)^k}\bigl(1-(-1)^{n-k}\bigr),\\
\end{gather*}
\begin{gather*}
 \sum_{k=1}^n{n\choose k}U_{k-1}(x)\left(\frac{2x^2-1}{2x^2\sqrt{x^2-1}}\right)^{k-1}\bigl(\omega(x)+(-1)^{n-k}\omega^{-1}(x)\bigr)\\
=\left(\frac{2x^2-1}{2x^2}\right)^{n-1}\sum_{k=0}^n{n\choose k} T_{2k+1}(x)\left(\frac{x}{(2x^2-1)\sqrt{x^2-1}}\right)^k\bigl(1-(-1)^{n-k}\bigr),\\
 \sum_{k=1}^n{n\choose k}U_{k-1}(x)\left(\frac{2x^2-1}{2x^2\sqrt{x^2-1}}\right)^{k-1}\bigl(1+(-1)^{n-k}\bigr)\\
=\left(\frac{2x^2-1}{2x^2}\right)^{n-1}\sum_{k=0}^n{n\choose k} T_{2k}(x)\left(\frac{x}{(2x^2-1)\sqrt{x^2-1}}\right)^k\bigl(1-(-1)^{n-k}\bigr),\\
\sqrt{x^2-1}\sum_{k=0}^n{n\choose k}U_{2k}(x)\left(\frac{x}{(2x^2-1)\sqrt{x^2-1}}\right)^{k}\bigl(1+(-1)^{n-k}\bigr)\\
=\left(\frac{2x^2}{2x^2-1}\right)^{n}\sum_{k=0}^n{n\choose k} T_{k}(x)\left(\frac{2x^2-1}{2x^2\sqrt{x^2-1}}\right)^k\bigl(\omega(x)-(-1)^{n-k}\omega^{-1}(x)\bigr),\\
	\sqrt{x^2-1}\sum_{k=0}^n{n\choose k} \frac{U_{2k}(x)}{\big(2x\sqrt{x^2-1}\big)^k}\bigl(1+(-1)^{n-k}\bigr)\\
	=\sum_{k=0}^n{n\choose k} \frac{T_{2k}(x)}{\big(2x\sqrt{x^2-1}\big)^k}\bigl(\omega(x)-(-1)^{n-k}\omega^{-1}(x)\bigr),\\
		(x^2-1)\sum_{k=1}^n{n\choose k} U_{2k-1}(x) \left( \frac{x}{(2x^2-1)\sqrt{x^2-1}}\right)^k\bigl(1+(-1)^{n-k}\bigr)\\
	=\left(\frac{2x^2}{2x^2-1}\right)^n\sum_{k=0}^n {n\choose k} T_{k}(x) \left(\frac{2x^2-1}{2x^2\sqrt{x^2-1}}\right)^k\bigl(1-(-1)^{n-k}\bigr),\\
	\sum_{k=1}^n {n\choose k} \frac{U_{2k-1}(x)}{\big(2x\sqrt{x^2-1}\big)^{k-1}} \bigl(\omega(x)+(-1)^{n-k}\omega^{-1}(x)\bigr)\\
	=2x\sum_{k=0}^n {n\choose k} \frac{T_{k}(x)}{\big(2x\sqrt{x^2-1}\big)^k} \bigl( 1-(-1)^{n-k}\bigr).
	\end{gather*}
\end{corollary}

\begin{remark} There also exist relations between balancing (Lucas-balancing) polynomials and Chebyshev polynomials of the third and fourth kinds, $V_n(x)$ and $W_n(x)$, respectively \cite{Mason}. These polynomials satisfy the recurrence \eqref{Cheb-Recur} with initial terms  $V_0(x) = 1$, $V_1(x) = 2x-1$, and $W_0(x) = 1$, $W_1(x) = 2x+1$. Well-known relations $V_n(x) = \sqrt{\frac{2}{1+x}}\,T_{2n+1}\!\left(\sqrt{\frac{1+x}{2}}\,\right)$ and $W_n(x) = U_{2n}\!\left(\!\sqrt{\frac{1+x}{2}}\,\right)$ give
$$
	\sqrt{\frac{1+x}{2}} V_n(x) = C_{2n+1}\!\left(\sqrt\frac{1+x}{18}\right), \qquad  
W_n(x) = B_{2n+1}\!\left(\sqrt{\!\frac{1+x}{18}}\right).
	$$ 
\end{remark}

\section{Fibonacci-Lucas identities via balancing polynomial relations}

The balancing and Lucas-balancing polynomials are closely related to the Fibonacci and Lucas numbers. Using this connection, in this section we obtain many Fibonacci-Lucas identities.

Let $F_n$ denote the $n$-th Fibonacci number and $L_n$ the $n$-th Lucas number, both satisfying the recurrence
$u_n=u_{n-1}+u_{n-2}, n \geq 2$, but with the respective initial terms $F_0=0$, $F_1=1$ and $L_0=2$, $L_1=1$.  
\begin{lemma} For $n \geq0$, the following relations hold:
	\begin{gather}\label{Fib-Luc-x=0.5}
	B_n\left(\frac{1}{2}\right)=F_{2n}, \qquad C_n\left(\frac{1}{2}\right)=\frac{L_{2n}}{2}.
	\end{gather}
\end{lemma}
\begin{proof}
The first formula follows from \eqref{B(x)-def} and the fact that sequence $(F_{2n})_{n\geq0}$ satisfies the recurrence $u_n=3u_{n-1}-u_{n-2}$, $n\geq2$. The proof of the second formula is omitted. 
\end{proof}

Using \eqref{Fib-Luc-x=0.5}, from Theorems \ref{Theo1_BC-gf} and \ref{Theo2_BC-gf} we immediately can obtain the following summation Fibonacci-Lucas identities.
\begin{corollary} For $n\geq1$, we have
	\begin{gather*}
	3F_{2n-1} = L_{4n-2}-4\sum_{k=1}^{n-1}F_{2k}L_{4n-4k-2},\\
	2F_{2n-1} -5F_{2n-2}= L_{4n-4}-4\sum_{k=1}^{n-1}F_{2k}L_{4n-4k-4},\\
	2F_{4n+2} -3F_{4n-2}= L_{2n}+L_{2n-2}+4\sum_{k=0}^{n-1}F_{4k+2}L_{2n-2k-2},\\
	2F_{4n} -3F_{4n-4}= 3L_{2n-2}+2\sum_{k=1}^{n-1}F_{4k}L_{2n-2k-2}.
	\end{gather*}
\end{corollary}

Next, by \eqref{Fib-Luc-x=0.5}, from Theorem \ref{Theo1_BC-egf} and \ref{Theo2_BC-egf} we obtain Fibonacci-Lucas identities involving binomial coefficients. 
\begin{corollary} Let $\alpha$ be the golden ratio, $\alpha=(1+\sqrt{5})/2$, and $\beta=(1-\sqrt{5})/2=-1/\alpha$.  For $n\geq1$, we have
	\begin{gather*}
	\sqrt{5}\sum_{k=1}^{n}{n\choose k}\!\left(\frac{2}{\sqrt5}\right)^k\bigl(1+(-1)^{n-k}\bigr)F_{2k}  = \sum_{k=0}^{n}{n\choose k}\!\left(\frac{2}{\sqrt5}\right)^k\bigl(1-(-1)^{n-k}\bigr)L_{2k},\\
	\sqrt{5}\sum_{k=1}^{n}{n\choose k}\!\left(\frac{2}{3\sqrt5}\right)^k\bigl(\alpha^2+(-1)^{n-k}\beta^2\bigr)F_{4k+2}  = \sum_{k=0}^{n}{n\choose k}\!\left(\frac{2}{3\sqrt5}\right)^k\bigl(\alpha^2-(-1)^{n-k}\beta^2\bigr)L_{4k+2},\\
\sqrt{5}\sum_{k=1}^{n}{n\choose k}\!\left(\frac{2}{3\sqrt5}\right)^k\bigl(1+(-1)^{n-k}\bigr)F_{4k}  = \sum_{k=0}^{n}{n\choose k}\!\left(\frac{2}{3\sqrt5}\right)^k\bigl(1-(-1)^{n-k}\bigr)L_{4k},
\end{gather*}
\begin{gather*}
\sqrt{5}\sum_{k=1}^{n}{n\choose k}\!\left(\frac{14}{9\sqrt5}\right)^k\bigl(\alpha^2+(-1)^{n-k}\beta^2\bigr)F_{2k}  = \left(\frac79\right)^n\sum_{k=0}^{n}{n\choose k}\! \left(\frac{6}{7\sqrt5}\right)^k\bigl(1-(-1)^{n-k}\bigr)L_{4k+2},\\
\sqrt{5}\sum_{k=1}^{n}{n\choose k}\!\left(\frac{14}{9\sqrt5}\right)^k\bigl(1+(-1)^{n-k}\bigr)F_{2k}  = \left(\frac79\right)^n \sum_{k=0}^{n}{n\choose k}\!\left(\frac{6}{7\sqrt5}\right)^k\bigl(1-(-1)^{n-k}\bigr)L_{4k},\\
	\sqrt{5}\sum_{k=1}^{n}{n\choose k}\!\left(\frac{6}{7\sqrt5}\right)^k\bigl(1+(-1)^{n-k}\bigr)F_{4k+2}  = \left(\frac{9}{7}\right)^n \sum_{k=0}^{n} {n\choose k}\! \left(\frac{14}{9\sqrt5}\right)^k\bigl(\alpha^2-(-1)^{n-k}\beta^2\bigr)L_{2k},\\
\sqrt{5}\sum_{k=1}^{n}{n\choose k}\!\left(\frac{2}{3\sqrt5}\right)^k\bigl(1+(-1)^{n-k}\bigr)F_{4k+2}  = \sum_{k=0}^{n} {n\choose k} \!\left(\frac{2}{3\sqrt5}\right)^k\bigl(\alpha^2-(-1)^{n-k}\beta^2\bigr)L_{4k},\\
\sqrt{5}\sum_{k=1}^{n}{n\choose k}\!\left(\frac{6}{7\sqrt5}\right)^k\bigl(1+(-1)^{n-k}\bigr)F_{4k}  = \left(\frac{9}{7}\right)^n\sum_{k=0}^{n} {n\choose k} \!\left(\frac{14}{9\sqrt5}\right)^k\bigl(1-(-1)^{n-k}\bigr)L_{2k},\\
\sqrt{5}\sum_{k=1}^{n}{n\choose k}\!\left(\frac{2}{3\sqrt5}\right)^k\bigl(\alpha^2+(-1)^{n-k}\beta^2\bigr)F_{4k}  = \sum_{k=0}^{n}{n\choose k} \!\left(\frac{2}{3\sqrt5}\right)^k\bigl(1-(-1)^{n-k}\bigr)L_{4k+2}.
\end{gather*}
\end{corollary}

The next result relates balancing and Lucas-balancing polynomials to Fibonacci and Lucas numbers \cite{Frontczak-AMS}.  
\begin{lemma}
	For $n \geq0$ and $s \geq1$,  the following hold:
	\begin{gather}\label{Fib-Luc}
	B_n\left(\frac{\varepsilon_{n}}{6}\, L_{s}\right)=\varepsilon_{n}^{n-1}\frac{F_{sn}}{F_{s}}, \qquad C_n\left(\frac{\varepsilon_{n}}{6}\,L_{s}\right)=\varepsilon^{n}_n\frac{L_{sn}}{2},
	\end{gather}
where $\varepsilon_n = \begin{cases}
1, & \text{if $n$ is even;}\\
i, & \text{otherwise.}
\end{cases}$
\end{lemma}
 The following corollary is an immediate consequence of Theorems \ref{Theo1_BC-gf}, \ref{Theo2_BC-gf} and \eqref{Fib-Luc}. 
\begin{corollary} For $n,m \geq0$, we have
	\begin{gather*}
2F_{sn}  = F_{s}L_{s(n-1)}+L_{s}F_{s(n-1)},\\
L_{s}\bigl(F_{s(2n+1)}-(-1)^sF_{s(2n-1)}\bigr) = F_{2s}\bigl(L_{s(2n+1)}+(-1)^sL_{s(2n-1)}\bigr),\\
2F_{2sn}-\bigl(L_{s}^2-(-1)^s2\bigr)F_{2s(n-1)} = F_{s}L_{s}L_{2s(n-1)},\\
	L_{s}\bigl(F_{sn}-(-1)^s\varepsilon_s F_{s(n-1)}\bigr)\\ = (-1)^s\varepsilon_s^{n+1}F_{s}L_{s(2n-1)}-(-1)^s\bigl(\varepsilon_{s}^2L_{2m}^2-\varepsilon_{s}L_{2m}-2\bigr)\sum_{k=1}^{n-1}\varepsilon_s^{n-k}F_{sk}L_{s(2n-2k-1)},\\
	2F_{sn}-\varepsilon_s(L_{2m}^2-(-1)^s2)F_{s(n-1)}\\ = (-1)^s\varepsilon_s^{n+1}F_{s}L_{2s(n-1)}+\bigl(\varepsilon_s^2L_{s}^2+\varepsilon_sL_{s}+2\bigr)\sum_{k=1}^{n-1}\varepsilon_s^{n-k} F_{sk}L_{2s(n-k-1)},
	\end{gather*}
	\begin{gather*}
	\varepsilon_s^n\bigl(2\varepsilon_sF_{s(2n+1)}-L_{s}F_{s(2n-1)}\bigr)\\ = F_{s}\bigl(\varepsilon_sL_{sn}+L_{s(n-1)}\bigr) - \bigl(\varepsilon_s^2L_{s}^2+\varepsilon_sL_{s}+2\bigr) \sum_{k=0}^{n-1}\varepsilon_s^kF_{s(2k+1)}L_{s(n-k-1)},\\
	2F_{s(2n+1)}-(L_{s}^2-(-1)^s2)F_{s(2n-1)} = F_{s}(L_{2sn}+(-1)^sL_{2s(n-1)}),\\
		\varepsilon_s^n(2F_{2sn}-(-1)^s\varepsilon_sL_{s}F_{2s(n-1)})\\=\varepsilon_sF_{s}L_{s}L_{s(n-1)}-(\varepsilon_s^2L_{s}^2+\varepsilon_sL_{s}+2)\sum_{k=1}^{n-1}\varepsilon_s^{k+1}F_{2sk}L_{s(n-k-1)},\\
	F_{2sn}-(-1)^sF_{2s(n-1)} = F_{s}L_{s(2n-1)}.
		\end{gather*}
\end{corollary}

Our last statement follows from Theorems \ref{Theo1_BC-egf}, \ref{Theo2_BC-egf} and \eqref{Fib-Luc}.  
\begin{corollary} For $n,s \geq0$, we have
	\begin{gather*}
	\sqrt5\sum_{k=1}^n\!{n\choose k}\! \left(\frac{\sqrt5F_s}{2}\right)^{n-k}\!\bigl(1+(-1)^{n-k}\bigr)F_{ks} = \sum_{k=0}^n\!{n\choose k}\! \left(\frac{\sqrt5F_s}{2}\right)^{n-k}\! \bigl(1-(-1)^{n-k}\bigr)L_{ks},\\
\sqrt5\sum_{k=1}^n\!{n\choose k}\! \left(\frac{\sqrt5F_{2s}}{2}\right)^{n-k}\bigl(\alpha^s+(-1)^{n-k}\beta^s\bigr)F_{(2k+1)s} \\
= \sum_{k=0}^n\!{n\choose k}\! \left(\frac{\sqrt5F_{2s}}{2}\right)^{n-k} \bigl(\alpha^s-(-1)^{n-k}\beta^s\bigr)L_{(2k+1)s},\\
\sqrt5\sum_{k=1}^n\!{n\choose k}\! \left(\frac{\sqrt5F_{2s}}{2}\right)^{n-k}\bigl(1+(-1)^{n-k}\bigr)F_{2ks}= \sum_{k=0}^n\!{n\choose k}\! \left(\frac{\sqrt5F_{2s}}{2}\right)^{n-k} \bigl(1-(-1)^{n-k}\bigr)L_{2ks},\\
\sqrt5L_{2s}^n\sum_{k=1}^n\!{n\choose k}\!\left(\frac{\sqrt5F_{2s}L_s}{2L_{2s}}\right)^{n-k} \bigl(\alpha^s+(-1)^{n-k}\beta^s)F_{ks} \\
=L_{s}^n\sum_{k=1}^n\!{n\choose k}\!\left(\frac{\sqrt5L_{2s}F_s}{2L_{s}}\right)^{n-k} \bigl(1-(-1)^{n-k})L_{(2k+1)s},\\
\sqrt5L_{2s}^n\sum_{k=1}^n\!{n\choose k}\!\left(\frac{\sqrt5F_{2s}L_s}{2L_{2s}}\right)^{n-k} \bigl(1+(-1)^{n-k})F_{ks} \\
=L_{s}^n\sum_{k=1}^n\!{n\choose k}\!\left(\frac{\sqrt5L_{2s}F_s}{2L_{s}}\right)^{n-k} \bigl(1-(-1)^{n-k})L_{2ks},\\
\sqrt5(2L_{s})^n\sum_{k=0}^n{n\choose k}\left(\frac{\sqrt5L_{2s}F_s}{2L_{s}}\right)^{n-k} \bigl(1+(-1)^{n-k})F_{(2k+1)s} \\
=(L_{2s})^n\sum_{k=0}^n{n\choose k} \left(\frac{\sqrt5F_{2s}L_{s}}{L_{2s}}\right)^{n-k} \bigl(\alpha^s-(-1)^{n-k}\beta^s)L_{ks},
\end{gather*}
\begin{gather*}
\sqrt5\sum_{k=0}^n{n\choose k}\left(\frac{\sqrt5F_{2s}}{2}\right)^{n-k} \bigl(1+(-1)^{n-k})F_{(2k+1)s} \\
=\sum_{k=0}^n{n\choose k} \left(\frac{\sqrt5F_{2s}}{2}\right)^{n-k} \bigl(\alpha^s-(-1)^{n-k}\beta^s)L_{ks},\\
\sqrt5L_{s}^n\sum_{k=1}^n{n\choose k}\left(\frac{\sqrt5L_{2s}F_{s}}{2L_{s}}\right)^{n-k} \bigl(1+(-1)^{n-k})F_{2ks} \\
=L_{2s}^n\sum_{k=0}^n{n\choose k} \left(\frac{\sqrt5F_{2s}L_s}{2L_{2s}}\right)^{n-k} \bigl(1-(-1)^{n-k})L_{ks},\\	
\sqrt5\sum_{k=1}^n{n\choose k}\left(\frac{\sqrt5F_{2s}}{2}\right)^{n-k} \bigl(\alpha^s+(-1)^{n-k}\beta^s)F_{2ks} \\
=\sum_{k=0}^n{n\choose k} \left(\frac{\sqrt5F_{2s}}{2}\right)^{n-k} \bigl(1-(-1)^{n-k})L_{(2k+1)s},
\end{gather*}
where $\alpha=(1+\sqrt{5})/2$ and $\beta=(1-\sqrt{5})/2$.
\end{corollary}

\bigskip
\hrule
\bigskip

\noindent 2010 {\it Mathematics Subject Classification}: 11B37, 11B39.

\bigskip
\noindent \emph{Keywords:} balancing polynomial, Lucas-balancing polynomial, Chebyshev polynomial, \break Fibonacci number, Lucas number, generating function.

\end{document}